\documentclass[11pt,openany,leqno]{article}
\usepackage[backref=page,colorlinks=true]{hyperref}
\usepackage{amsmath,amsthm,amsfonts,amssymb,amscd,url}
\usepackage[capitalize]{cleveref}

\usepackage{graphics,xcolor}
\numberwithin{equation}{section}
\input{xy} \xyoption{all}
\usepackage[latin1]{inputenc}

\usepackage{enumerate}
\usepackage{hyperref}
\usepackage{epsfig} 
\input{xy} \xyoption{all} 
\usepackage{mathrsfs}

\def \cI{{\mathcal I}}

\def \cP{{\mathscr P}}

\def \cR{{\mathcal R}}

\def \diag{\mathrm{diag\, }}

\def\qand{\quad \text{and}\quad}
\def\Diff{\mathrm{Diff}}

\def\A{\mathbb A}

\def\P{\mathbb P}
\def\C{\mathbb C}
\def\R{\mathbb R}
\def\N{\mathbb N}
\def\T{\mathbb T}
\def\D{\mathbb D}
\def\bS{\mathbb S}

\def\Z{\mathbb Z}

\def\cO{\mathcal O}

\def\det{\text{det}}
\def\exp{\text{exp}}
\def\Re{\text{Re}}

\def\leb{\mathrm{Leb}}
\def\Leb{\mathrm{Leb}}

\newtheorem{proposition}{Proposition}[section]
\newtheorem{theorem}[proposition]{Theorem}

\newtheorem{lemma}[proposition] {Lemma}

\newtheorem{theo}{Theorem}

\newtheorem{corollary}[theo]{Corollary}

\newtheorem{conjecture}[proposition]{Conjecture}

\theoremstyle{remark}

\newtheorem{remark}[proposition]{Remark}


\renewcommand*{\backref}[1]{}
\renewcommand*{\backrefalt}[4]{  \tiny 
  \ifcase #1 (\textbf{NOT CITED.})%
  \or    (Cited on p.~#2.)%
  \else   (Cited on p.~#2.)%
  \fi}
  
\begin{document}

\title{Coexistence of chaotic and elliptic behaviors among analytic,  symplectic diffeomorphisms of any surface}

\author{Pierre Berger\footnote{Partially  supported  by  the  ERC  project  818737  \emph{Emergence  of  wild  differentiable  dynamical systems.}}
}
\maketitle

\begin{abstract}
We show the coexistence of chaotic  behaviors (positive metric entropy) and elliptic  behaviors (integrable elliptic islands) among  analytic, symplectic diffeomorphisms in many  isotopy classes of any closed surface.  In particular this solves a problem introduced by F. Przytycki (1982).
%
%
\end{abstract}

\begin{theo}[Main result]\label{main}
For every analytic, symplectic and closed surface  $(S,\Omega)$, 
there is 
a symplectic, analytic map $f\in \Diff^\omega_\Omega (S)$  such that: 
\begin{enumerate}
 \item $f$ has positive metric entropy,
 \item $f$ displays elliptic islands.
\end{enumerate}
\end{theo}
A symplectic form $\Omega$ on an oriented surface is a nowhere-vanishing volume form. This defines a smooth measure $\Leb$ on $S$.    A mapping $f$ of $(S, \Omega)$ is \emph{symplectic} if it leaves the volume form $\Omega$  invariant. 
This is equivalent to say that it is orientation preserving and leaves $\leb$ invariant.  Then for $\Leb$  a.e. point $x\in S$ the limit $\Lambda(x):= \lim_{n\to \infty} \frac1n \log \|D_xf^n\|$ exists. The \emph{metric entropy} of $f$ is the mean of $\Lambda$.  Hence a dynamics has \emph{positive entropy} if  it is exponentially sensitive to the initial conditions with positive probability. 
An \emph{elliptic island} is a domain bounded by a smooth, invariant  curve on which the dynamics acts as an irrational rotation. There are many numerical experiments mentioning the coexistence these two phenomena for sympletic, analytic mappings, however so far no example was proved.

\begin{remark}
In the proof of \cref{main}, we will show moreover that $S$ without the support of $\Lambda$ is integrable: the dynamics is equal to the time one of a Hamiltonian flow. \end{remark}

\section{Introduction}
\subsection{History of the problem} 
This problem enjoys a long history. The first examples of mappings with positive entropy on any surface were discovered by Katok \cite{Ka79}. These examples are isotopic to  the identity. 
Then  Katok and Gerbert  \cite{GK82} obtained mappings with positive entropy on any surface in the  isotopy class of any pseudo-Anosov map.  Both constructions were
 smooth but not analytic. In \cite{Ge85}, Gerbert constructed real analytic symplectic pseudo-Anosov maps on any surface, which display positive metric entropy but not the coexistence with an elliptic island.  In  \cite{Pr82}, Przytycki built an example of conservative diffeomorphism of the torus with coexistence of an invariant region with positive entropy and an elliptic island. His construction was  infinitely smooth and not analytic. He addressed the problem of whether his construction could be generalized in the analytic class 
\cite[Rk1, P461]{Pr82}.  The issue of this problem was recalled as unclear by Liverani  in \cite[Rk 2.4 P3]{Li04} where a bifurcation of  Przytycki's example was studied. Note that \cref{main} solves in particular Przytycki's problem. 

In \cite{Go12},  Gorodetski proved that typical  examples of analytic symplectic surface maps are such that $\Lambda$ is positive on a set of  maximal Hausdorff dimension  ($=2$) and this coexists with   elliptic islands. However this leaves open a strong version of the positive entropy conjecture which asserts that ``a typical sympletic dynamics has positive metric entropy'' ($\Lambda$ is positive on a set of positive Lebesgue measure). A weaker   version of the positive entropy conjecture proposed by Herman \cite{He98} asserts the existence of  symplectic mappings $C^\infty$-close to the identity on the disk with positive metric entropy; it implies the density of surface maps with positive metric entropy among those with an elliptic cycle.   In \cite{BT19}, the Herman's positive entropy conjecture was proved with Turaev. Our proof used a  quotient similar to   the examples of Katok and Przytycki. 
During Katok's memorial conference 2019, in a conversation with Gorodetski and Kleptsyn,  I claimed that the  construction of \cite{BT19} should be useful to prove the following analytic counterpart of Herman's  positive entropy conjecture \cite{He98} and  even the next  analytic counterpart of  our main result with Turaev.
\begin{conjecture}\label{HeConj}
There exists an analytic and symplectic  perturbation of the identity of the disk with positive metric entropy.
\end{conjecture}
\begin{conjecture}\label{main2sd}
For every analytic and closed  symplectic surface  $(S,\Omega)$, for every analytic and symplectic  $f\in  \Diff^\omega_\Omega (S)$ which displays an elliptic periodic point, there are analytic and conservative  perturbations of $f$ with positive metric entropy.
\end{conjecture}
For the analogous strategy\footnote{For an introduction to the proof of \cite{BT19}, one could look at Arnaud's Bourbaki Seminar \cite{Ar19}.} of  \cite{BT19},  a first step toward the proof of Conjectures \ref{HeConj} and \ref{main2sd} is to prove the analytic counterpart of Przytycki's example. 
 
Following Gorodetski this step was not on reach in a short time, and  I bit with him the existence of such an example in a short time \footnote{More precisely the bit was that  someone would prove  within five years  the existence of an analytic symplectomorphism of the torus, isotopic to the identity,  with positive metric entropy and displaying an elliptic island.}. \cref{coroAprime,coro isotopy} solve this step: 
\begin{corollary}\label{coroAprime}
There exists an analytic and symplectic diffeomorphism $f$ of the closed disk displaying a stochastic island bounded  by  four heteroclinic bi-links which is robust  relative  link  preservation. \end{corollary}
Let us explain the meaning of the above statement.  We recall that a \emph{stochastic island} is a domain $\cI$ on which $\Lambda$ is positive $\leb$-a.e.   A \emph{bi-link} $C$ is a smooth circle equal to the union of two heteroclinic links $C= W^u(P)\cup \{Q\}= W^s(Q)\cup  \{P\}$ between saddle fixed points  $P$ and $Q$, see \cref{fig:island} \cpageref{fig:island}.  
Given a perturbation of the dynamics, the \emph{bi-link persists} if the union of the stable and unstable manifolds of the fixed points continue to form a smooth circle.  The island is \emph{robust  relative  link  preservation} if for every $C^2$-perturbation such that each of the bi-links persist, the domain bounded by the continuation of these bi-links is still a stochastic island. 
\medskip 
 
We can wonder also what are the isotopy classes of analytic, symplectic surface mappings which display coexistence phenomena. Our techniques enable (at least) to obtain the following:
 \begin{corollary}\label{coro isotopy}
For every analytic and closed symplectic surface  $(S,\Omega)$, 
for any isotopy class $C$, if
\begin{itemize}
\item $S$ is the $2$-torus for  the  isotopy class $C$  of the identity,
\item or $S$ is a surface of genus $\ge 0$ and $C$ is the isotopy  class of a pseudo-Anosov map of $S$,
\end{itemize}
then   there is a symplectic, analytic map $f\in \Diff^\omega_\Omega (S)$  of isotopy class  $C$,   
 such that $f$ has positive metric entropy and  displays elliptic islands. 
\end{corollary}

A natural problem would be to realize any isotopy class of surface diffeomorphisms by an analytic and symplectic dynamics displaying 
coexistence of positive metric entropy and elliptic islands. It seems that the techniques of this work together with the Nielsen-Thurston's classification of symplectic dynamics on  surface should lead to a solution of this problem. Another approach would be to prove \cref{main2sd} which would imply immediately a solution to the latter problem.

The proof of    \cref{main} is here completely self contained. 

\bigskip
\emph{I am grateful to A. Gorodetski and V. Kleptsyn for their  encouragements. I am thankful to  R. Krikorian and P. Le Calvez for nice conversations. I thank S. Biebler for his careful reading.  }

 \subsection{Idea  and structure  of the proof}
 All the proofs \cite{Ka79,Pr82,GK82,Li04,BT19} used bump functions to localize the surgery of  the dynamics in a subset of the manifold.  We recall that there is no analytic bump function. 
To deal with the analytic case, Gebert \cite{Ge85} showed that the pseudo-Anosov examples  of \cite{GK82} persist in a finite co-dimensional submanifold which must intersect the (infinite-dimensional) submanifold of analytic maps. However the examples of  \cite{Pr82,Li04,BT19}, 
displaying the sought coexistence,  persist actually along an infinite codimensional submanifold: one have to keep intact  heteroclinic links, and I do not see how to do this if the unperturbed map displaying the coexistence is not already analytic...
  Instead we propose another approach: 
  \begin{center}\emph{ 
We construct an analytic and symplectic extension of the surface punctured by several points, so that the extended surface remains diffeomorphic to the unpunctured surface, and the analytic continuation of the dynamics on the extended surface displays  elliptic islands.} 
  \end{center}
  \begin{figure}[h]
\begin{center}
\includegraphics[width=13cm]{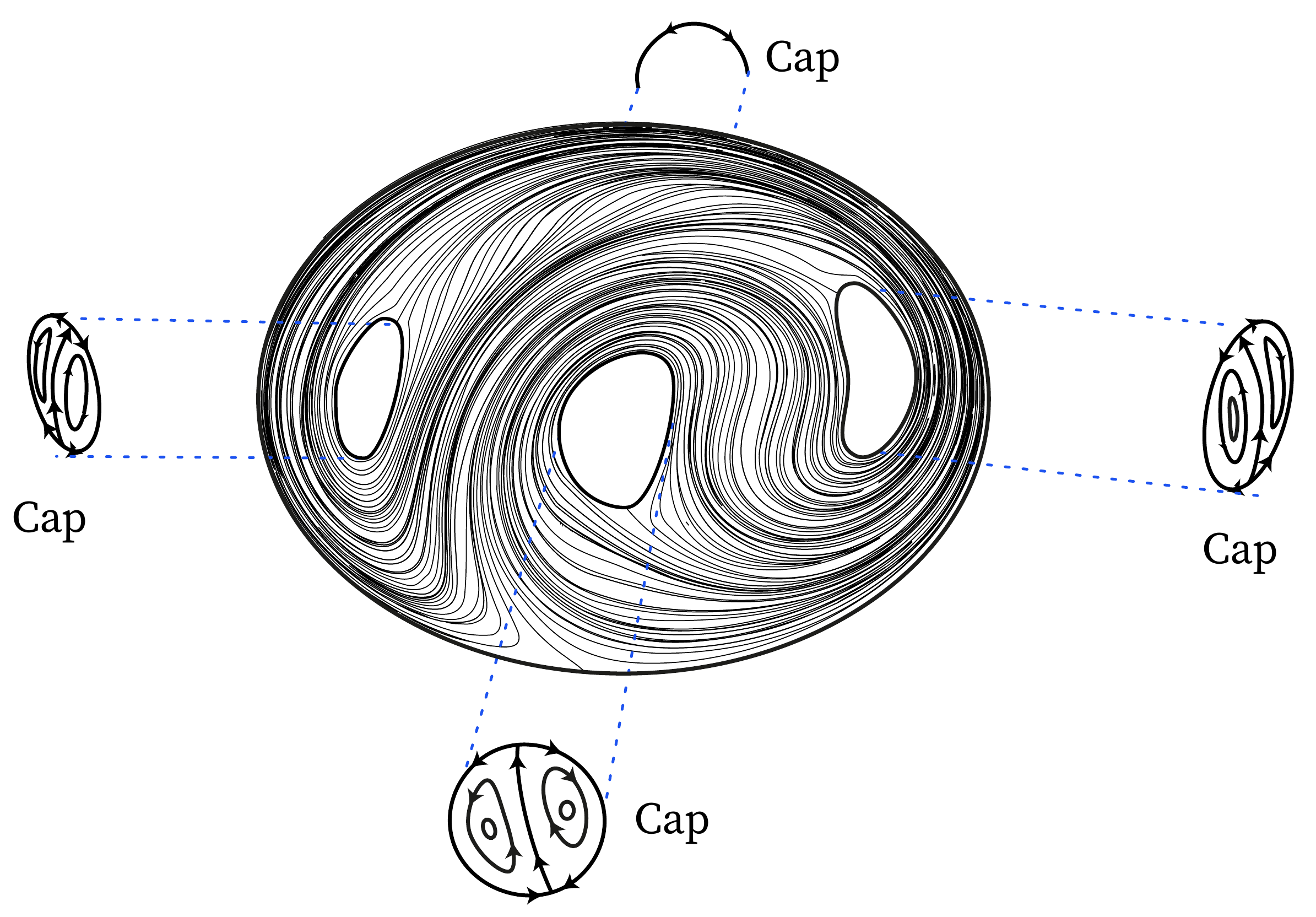}
\caption{Analytic and conservative dynamics on a sphere displaying coexistence of a stochastic region with elliptic islands.} \label{global }
\end{center}
\end{figure}
We will start with an analytic, conservative dynamics with positive entropy and then we will perform blow-up, quotient,  blow-down and connected sums, so that the analytic continuation of the dynamics is well defined after these operations and   displays the sought coexistence properties. 

For the proof of \cref{main}, we will start in  \textsection\ref{S1}  with a linear  Anosov map on the 2-torus, then we blow-up four of its fixed point \emph{\`a la} Przytycki  to define an analytic symplectic diffeomorphism of the 2-torus $\T^*$ without four disks, then we quotient it \emph{\`a la}  Katok to define  an analytic symplectic diffeomorphism of the 2-sphere $\bS^*$ without four disks in \textsection\ref{S2}. These steps were already performed in \cite{BT19} and are depicted in \cref{surgery Anosov}. Then we propose a new construction. 

First we regard the continuation of this  dynamics on an analytic extension $\hat \bS^*$ of  $\bS^*$ in \textsection\ref{S3}. Each component of  $\hat \bS^*\setminus  \bS^*$ is a \emph{collar}. This collar is diffeomorphic   to an annulus and equal to a halve neighborhood of two heteroclinic links.
In \textsection\ref{S4}  we glue two pieces of this annulus to obtain a collar which is a disk without two holes bounded by circle rotations, see \cref{cap surgery}.  
 In \textsection\ref{S5} we blow-down them to obtain a collar which is a disk containing two elliptic islands (the dynamics is actually integrable on the whole disk). Such are called \emph{cap's dynamics} on the disk.  This forms \emph{a cap} to recap any hole of the  sphere $\bS^*$ with  four holes.

This allows in \textsection\ref{S6} to prove the main theorem and the corollaries of its proof. In \textsection\ref{S61}, we start by proving \cref{main} when the surface is a sphere; the construction is depicted by \cref{global }. Following the number of recaped holes, coexistence phenomena are obtained on a  disk (which 
contains the stochastic island of \cref{coroAprime}), a cylinder or a pair of pants. The boundary of these can be glued together to form any closed symplectic surface, and so obtain \cref{main}. A careful study enables to obtain an analytic, sympletic  diffeomorphism of the torus isotopic to the identity, as wondered by Gorodetski and part of \cref{coro isotopy}. 

In \textsection\ref{SS4}, we prove the remaining part of \cref{coro isotopy} regarding surface mappings isotopic to a pseudo-Anosov map. 
  We will start   with the example of analytic pseudo-Anosov map of  \cite{Ge85}, which can represent any isoptopy class of orientiation preserving pseudo-Anosov maps (see also \cite{GK82}).  Then the punctured surface will be extended   following basically the same steps as in \textsection\ref{S3}.  The only difference is that the normal form \cite{Mo56} at the  saddle points is more general and that we will be working on a lifting of the previous construction. Caps will be replaced by a certain  \emph{generalized cap} given by \cref{generalized cap,Generalized cap}. The proof of the lemma follows the same lines as \textsection\ref{S4}-\ref{S5}. The Proposition  enables to  recap the surface with holes given by blowing up any periodic saddle orbit. The Lemma  enables to bound any cycle of heteroclinic links by a disk on which the dynamics is analytic and  integrable.

\section{Caps for  spheres with four holes} 
\subsection{A non-uniformly hyperbolic map on the torus  without four disks}\label{S1}
This step is depicted in \cref{surgery Anosov}[left-center]. 
\begin{figure}[h]
\begin{center}
\includegraphics[width=17cm]{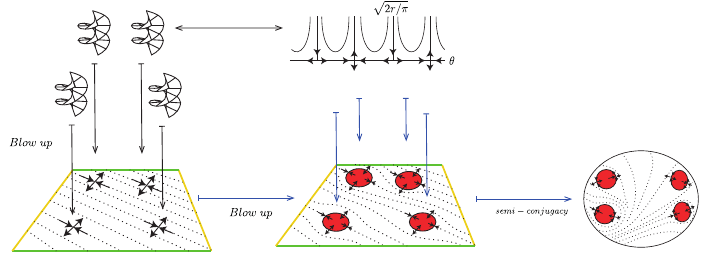}
\caption{Surgery on an Anosov map} \label{surgery Anosov}
\end{center}
\end{figure}

We start with the Anosov map $A(x,y)= (13\cdot x+8\cdot y, 8\cdot x+5\cdot y)$   which acts on the torus $\T^2:= \R^2/\Z^2$ endowed with the symplectic form $\Omega=dx\wedge dy$.  Let $R\in O_2(\R)$ and $\lambda>0$ be such that $A= R\times  \diag (\exp (\lambda), \exp(-\lambda)) \times R^{-1}$. 
The set $\cP:= \{0,  (1/2, 0) ,  (0,1/2),(1/2,1/2)\}$ is formed by four fixed points of the  Anosov map~$A$. We perform a symplectic and analytic  blow-up at each $P\in \cP$. 
Let $\epsilon>0$ be small.  The $\sqrt{\tfrac2\pi \epsilon}$-neighborhood $P+\D(\sqrt{\tfrac2\pi \epsilon})$ of $P\in \R^2/\Z^2$ is blown up to an annulus via the map:
\[\pi_P : (\theta,r)\in \R/2\Z \times [0, \epsilon  ] \mapsto P+R\times (\sqrt{\tfrac{2r}\pi} \cos (\pi \theta), \sqrt{\tfrac{2r}\pi} \sin (\pi \theta))\in P+\D(\sqrt{\tfrac{2 \epsilon}\pi}) \; .\]
These blow-ups are symplectic and analytic; they define a   new surface $\T^*$ as :
\[\T^*:= \left(\T^2\setminus \cP\right) \sqcup ( \cP \times \R/2\Z \times [0, \epsilon ])/\sim\]
with $\sim$ the equivalence relation spanned by:
\[\left(  z_1\in\T^2\setminus \cP \sim (P, z_2) \in   \cP \times \R/2\Z \times [0, \epsilon ]  \right)\Longleftrightarrow z_1=\pi_P(z_2) \; .\]
The surface $\T^*$ is analytic and symplectic; it is a torus without four disks. 
Also $A$ can be lifted to $\T^*$ as the map $A^*$ whose restriction to $\T^2\setminus \cP$ is $A$ and whose restriction to  a neighborhood of $ \cP \times \R/2\Z \times \{0\}$ is the time $1$ map of the Hamiltonian:
\begin{equation}\label{def:H}   H: (\theta,r)\mapsto     \tfrac{\lambda}\pi\cdot  r\cdot  \sin(2\pi \theta).\end{equation}
Indeed nearby each fixed point in $\cP$, in the coordinate induced by $R$, the map $A$ is the time one of the flow of the Hamiltonian  $H_1(x,y)=  \lambda \cdot x\cdot y= \lambda \cdot  
\sqrt{\tfrac{2r}\pi} \cos (\pi \theta)\cdot  \sqrt{\tfrac{2r}\pi} \sin (\pi \theta)=H(\theta,r)$.
\subsection{A non-uniformly hyperbolic map on the sphere  without four disks}\label{S2}

Note that $A$ is equivariant by the involution $-id$ on $\T^2$. 
The action of the involution on $\T^2\setminus \cP$ is free and $-id$ fixes each point in $\cP$.  
The involution $-id$ lifts to $\T^*$ as the involution  $J$ defined by:
\[ J| \T^2\setminus \cP= -id| \T^2\setminus \cP\qand J|\cP \times  \R/2\Z \times [0, \epsilon ]: (P,\theta,r)\mapsto (P,\theta+1,r)\; .
\]
 Let $\bS^* $ be the space  $\T^*$ quotiented by the involution $J$. This step is depicted in \cref{surgery Anosov}[center-right].  
Note that  $\bS^*$ is an analytic sphere without four disks. A neighborhood of the boundary of these four  holes is canonically parametrized by $\cP \times  \R/2\Z \times [0, \epsilon ]/J= \cP \times  \R/\Z \times [0, \epsilon ]$. 
 Also the  associated projection  $\pi^* : \T^* \to \bS^*$  is a $2$-covering. Since the symplectic form $\Omega$ is equivariant by $-id$, we can endow $\bS^*$ with the push forward of $\Omega$ by $\pi^*$ that we still denote by $\Omega$.  We notice that the dynamics $A^*$ descends to an analytic and symplectic dynamics $f^*$ on $\bS^*$. In other words, there is $f^*\in \Diff^\omega_\Omega (\bS^*)$ such that:
\[f^* \circ \pi^* = \pi^*\circ A^*\; .\]
  As   $A|\T^2\setminus \cP$ is a 2-covering of   $f^*| \bS^*\setminus \partial \bS^*$, the map  $f^*$ has positive metric  entropy. 
  
  \medskip 
  
  Now we shall embed the surface $\bS^*$ via a symplectic and analytic map   into a sphere so that the dynamics can be extended  to one which is analytic, symplectic and displays non-degenerates elliptic points.  As depicted in \cref{cap surgery}, this will be done first by implementing an explicit formula for the collar lemma, so that the each hole can be identified to the interior of a disk endowed with a dynamics at the neighborhood of the boundary. On the boundary it lies two saddle points; we will perform a surgery to glue a segment of the unstable branch of one to a segment of a stable branch of the other, so that the dynamics is extended to the disk without two disks, and finally we will blow-down each of these latter two disks to create two non-degenerated elliptic points. 

\subsection{Explicit collar lemma: the holes are surrounded by heteroclinic links}\label{S3}
  Let us now precise the dynamics of $f^*$ at the boundary of $\bS^*$.   By \cref{def:H}, nearby each component, the dynamics $f^*$  is equal to the time $1$ of the flow of the following  Hamiltonian: 
  \begin{equation}\label{def:Hbis}H: (\theta, r)\in \R/\Z \times [0, \epsilon )\mapsto 
    \tfrac \lambda \pi\cdot  r\cdot  \sin(2\pi \theta).\end{equation}
Observe that $\Omega= d\theta\wedge dr$ and  $H$ extend canonically to a neighborhood $V$ of $\R/\Z \times\{0\}$ in $\R/\Z \times\R$. Let $W\subset V$ be a neighborhood of $\R/\Z\times \{0\}$ in $V$ such that the  time $t=1$ of the Hamiltonian flow $\phi^t_H:W\to V$ of $H$ is well defined.  For $\pm\in \{-, +\}$, let  $V^\pm :=V\cap \R/\Z\times \R_\pm$  and  $W^\pm :=W\cap \R/\Z\times \R_\pm$. 
  On the boundary $\partial W^-=\partial W^+= \R/\Z \times\{0\}$, the map $\phi^{1}_H$ displays two saddle fixed points $Q=(0,0)$ and $Q':=(1/2, 0)$, so that $W^s(Q)\setminus \{Q\}=W^u(Q')\setminus \{Q'\}= \partial W \setminus \{Q,Q'\}$. 
 In particular $\partial W^\pm$ is a bi-link. 
 
Note  that $\bigsqcup_{\cP} W^+$ and $\bigsqcup_{\cP} V^+$ are neighborhoods of $\partial \bS^*$ in $\bS^*$. So we can extend  the surface $\bS^*$ by gluing  canonically   $ {\cP}\times  V$ of $V$ at  $ {\cP}\times  V^+\subset  \bS^*$. This defines an open surface $\hat \bS^*$  which contains $\bS^*$ and such that $\hat \bS^*\setminus int\, \bS^* $ is equal to $\bigsqcup_{\cP} V^-$.  On a neighborhood of $\bS^*$ in $\hat \bS^*$, the map $f^*$ extends analytically to a map denoted by $\hat f^*$ and whose restriction to $\cP\times W^-$ is  $\phi^{1}_H$.

\subsection{From holes surrounded by heteroclinic links to holes  surrounded by rotations }\label{S4}
The idea is to shape $W^-$ as in   \cref{cap surgery} [left] to perform  the surgery depicted in \cref{cap surgery} [left-center]. 
 \begin{figure}[h]
\begin{center}
\includegraphics[width=17cm]{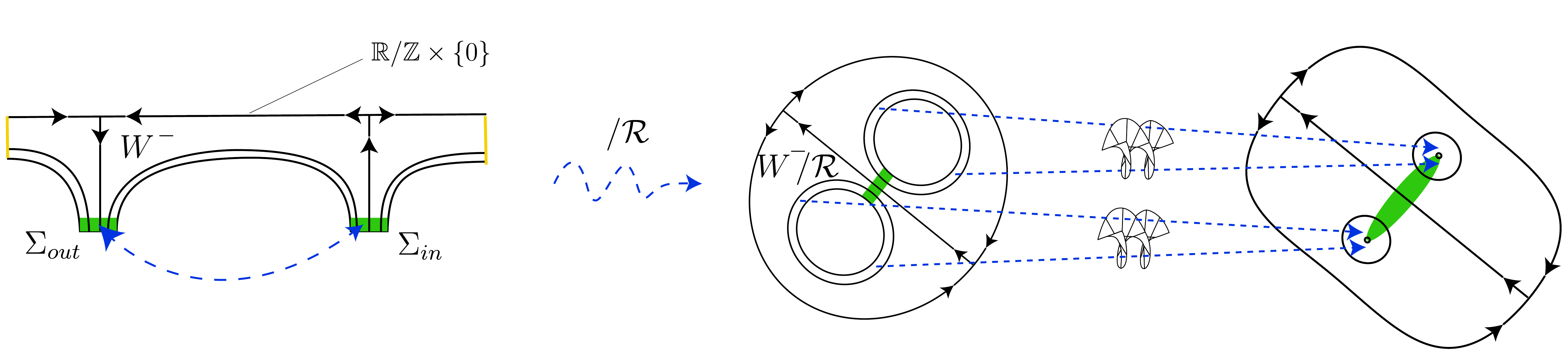}
\caption{Making an integrable cap  by gluing the green rectangles together and then blowing down.} \label{cap surgery}
\end{center}
\end{figure}

Recall that $H$ extends to $V\supset W$.  For $\eta>0$ small, we can shrink $W$ to have  $W^-$ of the form:
\[W^-:= \{(\theta,r)\in \R/\Z\times 
[-\eta, 0]:
|H(\theta,r)|\le \eta^3 \}\; .\] 
By \cref{def:Hbis},  the boundary of $W^-$ is formed by three curves;  the union of two of them is:
\[\Sigma:=\{\theta\in \R/\Z: 
|H(\theta,-\eta)|\le \eta^2 \}\times \{-\eta\}\]
Let $\Sigma_{out}$ be the component of $\Sigma$  which intersects $W^u_\eta(Q)=\{0\}\times [-\eta,0]$. Let $\Sigma_{in}$ be the component of $\Sigma$  which intersects 
$W^s_\eta(Q')=\{\tfrac12 \}\times [-\eta,0]$.  To perform the surgery depicted in \cref{cap surgery} [left-center],  we define:
\[\phi:   \bigcup_{t\in [-1,0]} \phi_H^t(\Sigma_{out})\to  \bigcup_{t\in [0,1]} \phi_H^t(\Sigma_{in})\; ,\]
such that for $(\theta,r)\in \Sigma_{out}$, $\phi(\theta,r)= \phi_{H}^{1}(\tfrac12-\theta,r)$ and for $t\in [0,1]$, 
$\phi\circ  \phi^{t-1}_H(r,\theta)= \phi^{t}_H(\tfrac12-\theta,r)$. 

Note that the map $\phi$   is  analytic and  symplectic with range $\bigcup_{t\in [0,1]} \phi_H^t(\Sigma_{in})$. Moreover it respects $H$.   
  Let $W^-/\cR $ be equal to $W^-$ quotiented by the equivalence relation $\cR$  spanned by:
  \[\forall 
 (r,\theta)\in  \bigcup_{t\in [-1,0]} \phi_H^t(\Sigma_{out}),
\text{ and }
  (r',\theta')\in  \bigcup_{t\in [0,1]} \phi_H^t(\Sigma_{in}),
 \quad  (r,\theta)\cR  (r',\theta')\quad  \text{iff } ( r,  \theta)= \phi( r', \theta').\] 
  As $\phi$ is an analytic diffeomorphism which
sends the intersection of its domain with $\partial W^-\setminus  \Sigma$  into $\partial W^-\setminus  \Sigma$,  the space  $W^-/\cR$  is an analytic surface with boundary. Moreover, we notice that $W^-/\cR$ is equal to the closed  disk $\bar \D$ without two  disks $\D_+$  and $\D_-$ , as depicted in \cref{cap surgery} [center]:
\[W^-/\cR=  \bar \D \setminus (\D_+ \sqcup \D_-)\]
 This surgery respects the symplectic form $\Omega$ and the Hamiltonian $H$, which remains analytic on $W^-/\cR$. Hence they push-forward  to a symplectic form and a Hamiltonian still denoted by $\Omega$ and $H$.  Note also that the analytic continuation of $f$ on  $W^-/\cR$ is still equal to the time one map of the Hamiltonian $H$.   
\subsection{Blowing down holes surrounded by rotations}\label{S5}
Now we would like to   blow-down the circles  $ (\partial \D_- \sqcup  \partial \D_+)$ to fixed points $P_-$  and $P_+$.   These blow-downs will construct a pair of disks depicted in  \cref{cap surgery} [Center-Right].
  In order to do so, we first observe that on $  \partial \D_+\sqcup \partial \D_-$, the function $H$ is constant (equal to resp. $\eta^3$ and $-\eta^3$) and its   symplectic gradient does not vanish. So we can apply the classical action-angle coordinate change:
\begin{lemma}\label{euclidean coordinate}Let  $V_0$ be a neighborhood $V_0$ of  $C=\R/\Z\times \{0\}$ in $\R/\Z\times \R^+$ and $H: V_0\to \R$ be an analytic  Hamiltonian constant on $C$  and whose differential does not vanish on $C$. 
Then there are $\delta>0$, a neighborhood $V'_0\subset V_0$ of 
$C$  and an analytic and symplectic map $\psi: V'_0
\to \R/\Z\times [0,\delta]$ so that $H\circ \psi^{-1}$ 
  sends the orbit of each point of $V'_0$  to a horizontal circle $\R/\Z\times \{\rho\}$.
\end{lemma}
\begin{proof}[Proof of \cref{euclidean coordinate}]
For $\delta'>0$ small enough, the section   $\Pi:=\{0\}\times [0,\delta']$ is transverse to the orbits of the Hamiltonian  $H$.  Every point $z$ in $\Pi$ is periodic of a period $T(z)$. The union of the orbits of points in $z\in \Pi$ is equal to a neighborhood $V'_0\subset V_0$ of $C$. 
Let us define the following flow box: 
\[\psi_0: V'_0\to \R\times \Pi/\sim \quad \text{with }(t,z)\sim ( t+kT(z),z),\quad k\in \Z\; .\]
In these coordinates the flow $\phi^t_H$  of $H$ is the translation by  $(t,0)$. We shall reshape the range of $\psi_0$ so that it equals $\R/\Z \times [0,\delta]$ with $\delta =\Leb(\R\times \Pi/\sim)$. To this end we consider   a primitive $\hat T$ of $T|\Pi$ which vanishes at $0\in \Pi\cap C$:
  \[\hat T  : (0,r) \in \Pi=\{0\} \times [0,\delta'] \mapsto   \int_0^{r} T(0,x)dx\; . \]
Note that  $\delta:=\hat T(\delta')$. Let  $\zeta:  \R\times \Pi/\!\! \sim\; \to  \R/\Z\times [0,\delta  ]$ be defined by:
\[\zeta(z,t)  =( - t/T (z),\hat T(z))  \; .
\]
We observe that $\zeta$ is analytic and symplectic. Thus $\psi =
 \zeta\circ \psi_0$   is an analytic and symplectic map from $V'_0$  onto $\R/2\Z\times [0,\hat T(\delta) ]$ which sends each orbit to a horizontal circle $\R/\Z\times \{\rho\}$. 
\end{proof}

Thus  for every $\pm\in \{-,+\}$, there exist a neighborhood $V_\pm$ of $\partial\D_\pm$ endowed with analytic and symplectic coordinates $\psi_\pm: V_\pm\to \R/\Z\times [0,\delta]$  such that  $H\circ \psi^{-1}_\pm(\theta,r)= h_\pm(r)$  for an analytic maps $h_\pm$.   So we can perform a blow-down of the circle $\partial \D_\pm$ to a fixed point  $P_\pm$. This blow-down sends $V_\pm$ to a disk of radius $\sqrt{ \delta/\pi}$ and  the dynamics on this disk is the time one map given by the    Hamiltonian $ P_\pm+(x,y)\mapsto h_\pm( x^2+y^2)$, which is indeed integrable and displays an elliptic island.  
 
\section{Application of the construction 
}\label{S6}
\subsection{Proof of \cref{main,coroAprime}}\label{S61}
Steps \textsection\ref{S1} and \textsection\ref{S2} constructed an analytic diffeomorphism $f^*$   on the sphere $\bS^*$ without four holes with positive metric entropy. 
At step \textsection\ref{S3}, we saw that this dynamics could be extended at the neighborhood $\equiv \R/\Z\times (-\epsilon,\epsilon)$ of boundary $\equiv \R/\Z\times \{0\}$  of each of these holes by the time one map  of the Hamiltonian:
\[H(\theta, r)=
  \tfrac \lambda\pi\cdot  r  \cdot \sin(2\pi \theta).\]
 
Steps   \textsection\ref{S4} and \textsection\ref{S5}  constructed a cap for any of the holes of $\bS^*$. This is an  analytic  dynamics $\hat f$  on the disks which displays two elliptic fixed points and which is integrable: it is  the time one map of an Hamiltonian. Moreover this Hamiltonian coincide  at the neighborhood boundary  the disk $\approx \R/\Z\times \{0\}$ with the  Hamiltonian $H$.  As a matter of fact, we can fillup any hole of $\bS^*$ endowed with $f^*$, by a a disk endowed with $\hat f$. 

Let us perform  surgeries with these two objects in order to deduce \cref{main,coroAprime}. 
\begin{proof}[Proof of \cref{main}]
\underline{Case where $S$ is the sphere.}
In the above construction, we recap each of the four  holes of $\bS^*$ with a disk, and we endow $\bS^*$ with $f^*$ and each four disks with the cap's dynamics $\hat f$. We obtained an analytic, sympletic map of the sphere with positive metric entropy and displaying four elliptic islands. 

\underline{Case where $S$ is the torus.}
In the above construction, we fill up two  holes of $\bS^*$  
with two disks, and we endow $\bS^*$ with $f^*$ and each disk with the cap's dynamics $\hat f$. This defines a symplectic  dynamics $f^\A$ of the annulus  $\A$ with positive entropy and four elliptic points (two in each cap).  A neighborhood of $\partial \A$ in $\A$ is 
diffeomorphic to $\bigsqcup_{\pm\in \{+,-\}}\R/\Z\times [\pm 1,\pm (1-\epsilon)]$ and in these coordinates  
the dynamics $f^\A$ is the time one of the flow of $H(\theta, r\pm1)$ with:
 \[H(\theta, r\pm1)=   \tfrac \lambda\pi\cdot  r  \cdot \sin(2\pi \theta).\]
 We  glue both boundaries of $\A$ by:
$(\theta,r+1)\sim(\theta, r-1)$  for $r$ small.   This defines an analytic and symplectic map  on the torus   with a priori  non-trivial isotopy class.

\underline{Case where $S$ is a surface of higher genus.}
In the above construction, we recap only one holes of $\bS^*$  
with a disk to form a pair of pants $\P$: a disk with  two holes. We endow $\bS^*$ with $f^*$ and the disk with the cap dynamics $\hat f$. This defines a symplectic  dynamics $f^\P$ on $\P$. We recall that every closed, oriented surface $S$ of genus $\ge 2$ displays a pants decomposition. We glue canonically (as above) the pants at their boundaries to obtain the sought dynamics. 
\end{proof}
\begin{proof}[Proof of \cref{coro isotopy} for $S$ equal to the torus and $f$ isotopic to the identity.] 
We constructed above a symplectic and analytic maps $f^\A$ on the annulus $\A $ such that  at the neighborhood  $\bigsqcup_{\pm\in \{+,-\}}\R/\Z\times [\pm 1,\pm (1-\epsilon)]$ of the boundary   $\partial\A$  the dynamics $f^\A$ is the times one of the flow of $H(\theta, r\pm1)$ with:
 \[H(\theta, r\pm1)=   \tfrac \lambda\pi\cdot  r  \cdot \sin(2\pi \theta).\]
 We saw that if we  glue both boundaries of $\A$ by $(\theta,r+1)\sim(\theta, r-1)$  for $r$ small, the dynamics is \emph{a priori} in a non-trivial isotopy class.  To vanish this isotopy class, the idea is to glue $f^*$ with its inverse $f^{*-1}$.  At the boundary of $\A$, the map $f^{*-1}$ is the time one of the flow of  $-H(\theta, r\pm1)= H(-\theta, r\pm1)$. 
So we glue two copies $\A_1 $ and $\A_2 $
of $\A$ along their respective boundaries  
by $(\theta_1,r_1+1)\sim(-\theta_2, r_2-1)$ and $(\theta_1,r_1-1)\sim(-\theta_2, r_2+1)$ for  $r_1,r_2$ small and any $\theta_1,\theta_2$.  
The space $\hat \A$ obtained is a $2$-covering of a Klein bottle, which is a torus. The dynamics induced by $f^*$ and $f^{*-1}$  on $\hat \A$ is analytic, symplectic and isotopic to the identity.  Moreover the dynamics on  the torus displays the coexistence phenomenon. 
\end{proof}
\begin{proof}[Proof of \cref{coroAprime}]
We first start with the above analytic diffeomorphism $f^*$   on the sphere $\bS^*$ with four holes. We recap three holes with caps and we  endow each of them with the dynamics $\hat f$. This defines an analytic and symplectic map $f^\D$ on the disk $\D$. Note that the disk is not endowed with its standard sympletic form, but using \cite{DM90}, we can analytically conjugate it to one which leaves invariant the standard symplectic form on $\D$. The image $\cI$ of $\bS^*$ in $\D$  is depicted  Fig. \ref{fig:island}. Therein  the Lyapunov exponent function $\Lambda$ is $\leb$ a.e.  equal to a positive constant. 
  \begin{figure}[h!]
	\centering
		\includegraphics[width=5cm]{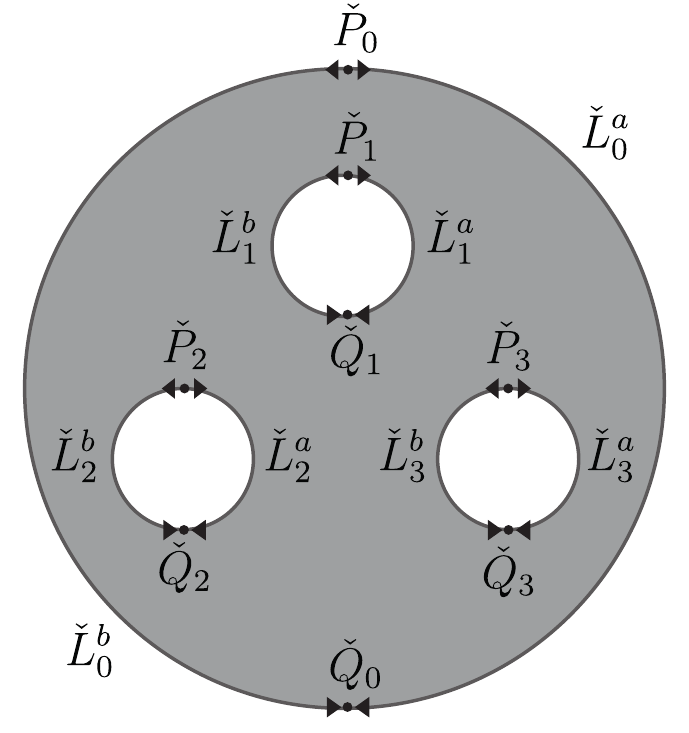}
\caption{Stochastic island $\cal I$ in grey.}\label{fig:island}
\label{fig:oldcoord}
\end{figure}
In the sense of \cite{BT19} (inspired from \cite{Ka79, AP09,Pr82}), the set $\cI$  is called  a \emph{stochastic island}. This means that $\cI$ is  a disk with three holes;  and that the boundary of $\cI$ is formed by four pairs of  heteroclinic  bi-links $\{(\check L_i^a, \check L^b_i): 0\le i\le 3\}$. Each $\check L_i^a\cup \check L_i^b$ is a smooth circle included
in the stable and unstable manifolds of hyperbolic fixed points   $\check P_i$ and $\check Q_i$ respectively:
\[\check L_i^a\cup \check L_i^b \subset   W^u(\check P_i; f^\D)\cup W^s(\check Q_i; f^\D)\; .\]
For every $f$ which is $C^1$-close to $f^\D$, for every $0\le i\le 3$, the \emph{hyperbolic continuations} $P_i$ and $Q_i$ of $\check P_i$ and $\check Q_i$  are  uniquely defined hyperbolic fixed points for $f$.   If $\{W^u(P_i; f)\cup W^s(Q_i;f): 0\le i \le 3\}$ form four heteroclinic bi-links $\{L_i^a\cup  L^b_i: 0\le i\le 3\}$  close to $\{\check L_i^a\cup  \check L^b_i: 0\le i\le 3\}$, then we say that the \emph{bilinks are persistent for the perturbation $f$}.

Then the next proposition implies \cref{coroAprime}.  
\end{proof}
\begin{proposition}[{\cite[prop. 2.1]{BT19}}]\label{KPAP}
For every conservative map $f$ which is $C^2$-close to $f^\D$ and for which the bi-links are persistent, then the continuations of these bilinks bound a stochastic island. In particular, the metric entropy of $f$ is positive.  
\end{proposition}
\subsection{Proof of  \cref{coro isotopy}}\label{SS4}
It remains  the case of mappings isotopic to pseudo-Anosov maps. 
To carry them we will need to generalize the construction of cap to prove:
\begin{proposition}\label{generalized cap}Let $U\subset \R^2$ be a neighborhood of $0$.  Let  $f: U\to \R^2$ be an analytic and symplectic map with $0$ as a saddle fixed point with positive eigenvalues. Then there exist $\rho>\rho'>0$ arbitrarily small, an 
analytic and symplectic map  $\phi: \D (\rho)\setminus \D(\rho')\to \D (\sqrt{\rho^2-\rho'^2})\setminus \{0\}$ such that on:
\[\hat U:= (U\setminus \{0\})\sqcup \D(\rho)/\sim \quad 
\text{with } U\ni u\sim d\in \D(\rho) \text{ iff } u=\phi(d)\; ,\]
 the map $f|U\setminus \{0\}$ extends to a symplectic and analytic map $\hat f$ on $\hat U$ which leaves invariant the disk $ \D(\rho')$ and on which its   restriction   is integrable and displays three elliptic fixed points. \end{proposition}  

\begin{proof}[Proof of \cref{coro isotopy} for  $f$ isotopic to a pseudo-Anosov map.] 
Let $(S, \Omega)$ be a symplectic orientable, closed surface. Then by \cite{GK82,Ge85},   any orientation preserving pseudo-Anosov isotopy class is represented by  an analytic and symplectic map   $f$ on $S$. 
\begin{lemma}
The map $f$  displays a  hyperbolic periodic cycle $(P_i)_{i\in \Z_q}$ with positive eigenvalues.\end{lemma}
\begin{proof}
As $f$ has positive topological entropy, it displays a horseshoe \cite{Ka80} with at least two rectangles. There are two possibilities:
Either one of these  rectangles is not rotated by the induced dynamics, and so we get immediately a saddle  periodic cycle with positive eigenvalues. Or both rectangle are rotated by a half turn. Then we can compose the induced dynamics by these two rectangles to obtain a  hyperbolic periodic cycle with positive eigenvalues.
\end{proof}
Let $U_0$ be neighborhood of $P_0$  and let $U_i:= f^{i}(U_0)$ for every $i$. 
 We assume $U_0$ small enough so that $U_0\cup U_q$ can be identified to a  subset of $\R^2$. Then each $U_i$, $1\le i\le q-1$, can also be identified to subset of $\R^2$ using the diffeomorphism $f^{q-i}: U_i\to U_q$. In these identifications,
$U_1\equiv\cdots\equiv U_{q-1}\equiv U_q$,  $f|U_i\to U_{i+1}$ is the identity for $1\le i\le q-1$ and $f|U_{0}\equiv f^q|U_0$. 
 
We apply \cref{generalized cap} to $f^q|U_0$. This blows up $P_0$ to a disk, and so $U_0$ and $U_q$ to $\hat U_0$ and $\hat U_q$.
Furthermore we can lift $f: U_0\to U_1$ to $\hat f: \hat U_0\to \hat U_1$ so that $\hat f^q|\hat U_0$ displays three elliptic islands.
  Using the above  identifications, we  also blow up each $U_1\equiv\cdots\equiv U_{q-1}\equiv U_q$ to $ \hat U_1\equiv\cdots\equiv \hat U_{q-1}\equiv\hat U_q$. We lift each $f|U_i\equiv id$ to $\hat f|\hat U_i\equiv id$ for $1\le i\le q-1$.   

All these blow-ups can define a blow up $\hat S$ of $S$ along the orbit $(P_i)_{i\in \Z_q}$, and a lifting $\hat f$ of $f$ which displays the sougth properties. 
\end{proof} 

\begin{proof}[Proof of \cref{generalized cap}]
By \cite{Mo56}, there exists analytic and symplectic coordinates of a  neighborhood  of $0$ for which the dynamics is of the following form with $\lambda$ an analytic function:
\[f (x,y)= (\exp(\lambda(x\cdot y)) \cdot x, \exp(-\lambda(x\cdot y))\cdot y), \quad \text{with } \lambda(x\cdot y) >0\; .\]
Let $\Lambda$ be an integral of the function $\lambda$ so that $\Lambda(0)=0$. 
Note that $f$ is the time one of the flow of the Hamiltonian  $H_1:(x,y)\mapsto \Lambda (x\cdot y)$.  Let us  follow the same lines as in \cref{S3,S4,S5}. The difference is that the function $  \lambda  $ here is not constant. Also we will not perform the quotient $\R/2\Z\times \R\to \R/\Z\times \R$. Hence   basically, we will construct a 2-covering of the previous~cap.

First, we perform a symplectic and analytic  blow-up of $0$.   
For $\epsilon>0$ small, let $r_\epsilon:= \sqrt{\tfrac2\pi \epsilon}$. 
The $r_\epsilon$-neighborhood $\D(r_\epsilon)$ of $0$ is blown up to an annulus  via the map:
\[\pi_0:    (\theta,r)\in \R/2\Z \times [0, \epsilon  ]  \mapsto   (\sqrt{\tfrac{2r}\pi} \cos (\pi \theta), \sqrt{\tfrac{2r}\pi} \sin (\pi \theta))\in  \D(r_\epsilon)\; .\] 
This blow up defines a symplectic and analytic  annulus $\D^*$ as :
\[\D^*:= \left(\D(r_\epsilon)\setminus \{0\}\right) \sqcup   \R/2\Z \times [0, \epsilon ]/\pi_0\; .\] 
Note that $H_1$ lifts to $\R/2\Z \times [0, \epsilon ] $ as the mapping:
\[H: (\theta, r)\mapsto \Lambda(\tfrac r\pi \sin(2\theta))\]
Observe that $\Omega= d\theta\wedge dr$ and  $H$ extend canonically to $V:= \R/2\Z \times (-\epsilon, \epsilon )$. 
On  $C= \R/2\Z\times \{0\}$, the dynamics  displays four saddle fixed points $Q_i=(\tfrac i 2,0)$ with $i\in \Z/4\Z$, so that $ \bigcup_{i=1}^2 W^s(Q_{2i})\cup \{Q_{2i+1}\} = 
\bigcup_{i=1}^2 W^u(Q_{2i+1})\cup \{Q_{2i}\}= C$. 
 In particular $C$ consists of four  heteroclinic links. Thus we can   conclude by applying the next lemma with $k=2$. \end{proof}
\begin{lemma}[Generalized cap] \label{Generalized cap}
Let $H$ be an analytic Hamiltonian defined on a neighborhood $V$ of $C= \R/2\Z\times \{0\}$ in $\R/\Z\times \R$  such that $C $ is an union of $2k$-heteroclinic links: $C= \bigcup_{i\in \Z_k} W^s(Q_{2i}) \cup \{Q_{2i+1}\}= \bigcup_{i\in \Z_k} W^u(Q_{2i+1}) \cup \{Q_{2i}\}$. 
Then there exists there exists $\delta>0$ and an analytic map $\phi$ from $\R/2\Z\times (-\delta,0]$ to a neighborhood of $\partial \D$ in $\bar \D$ such that $H\circ \phi^{-1}$ extends to an analytic function on $\bar \D$ which displays exactly $k+1$ critical points in $\D$  with definite Hessian. 
\end{lemma}
\begin{proof} We depict the construction for $k=2$ in \cref{cap surgeryS} [left-center]. For $k=1$, this lemma implies \textsection \ref{S4} and \ref{S5}; its proof is similar. 
 \begin{figure}[h]
\begin{center}
\includegraphics[width=17cm]{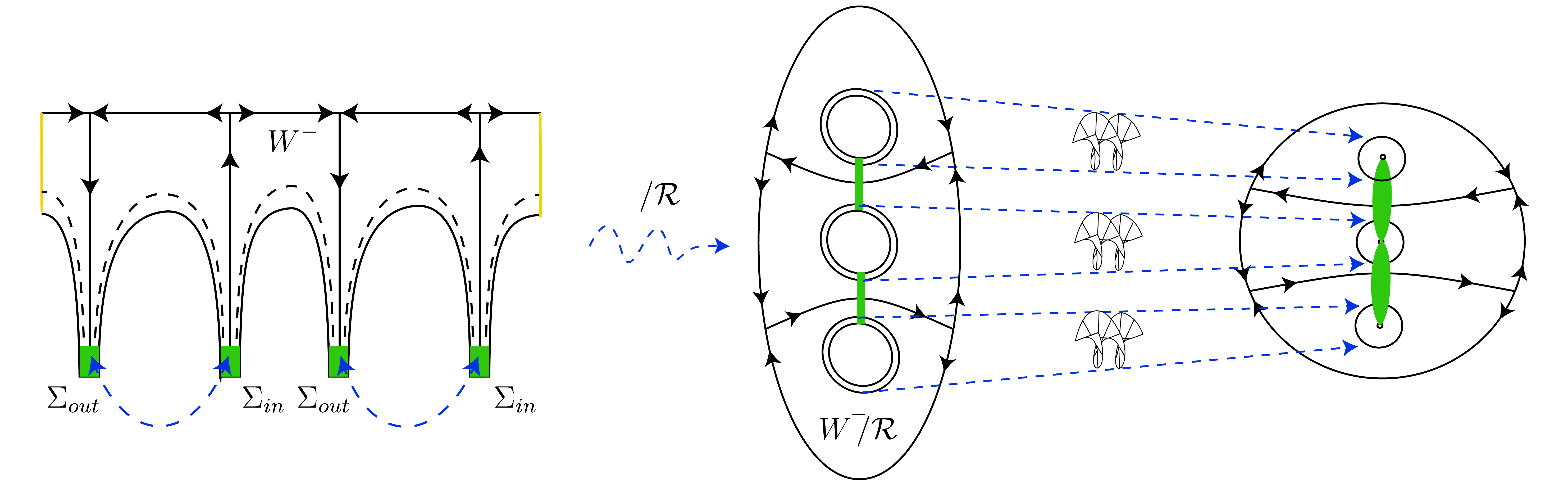}
\caption{Making an integrable generalized cap by surgery with $k=2$.} \label{cap surgeryS}
\end{center}
\end{figure}

On $C$ the function $H$ must be constant; let us assume it equal to $0$.  For $\eta>0$ small, we define:
\[W^-:= \{(\theta,r)\in \R/2\Z\times 
[-\eta, 0]:
|H(\theta,r)|\le \eta^3\qand |r| \le \eta\}\; .\] 
 The boundary of $W^-$ is made by $4k+1$  curves (see \cref{cap surgery}[left]);  $2k$ of them form:
\[\Sigma:=\{\theta\in \R/2\Z: 
|H(\theta,r)|\le \eta^2 \}\times \{-\eta\}\; .\]
For every $ i\in \Z_k$, let $\Sigma_{out,i}$ be the component of  $\Sigma$  which intersects $W^u_{2\eta}(Q_{2i})$ and   $\Sigma_{in,i}$ be the component of  $\Sigma$  which intersects $W^s_{2\eta}(Q_{2i+1})$. On each $\Sigma_{out,i}$, the restriction $H|\Sigma_{out,i}$ is a diffeomorphism onto $[-\eta^3,\eta^3]$. Thus there is a canonical parametrization of $\Sigma_{out}$ with  $[-\eta^3,\eta^3]\times \Z_k$. Likewise  there is a canonical parametrization of $\Sigma_{in}$ with  $[-\eta^3,\eta^3]\times \Z_k$.  Let $\phi_H^t$ the Hamiltonian flow of $H$. 
 To perform the surgery depicted in \cref{cap surgeryS} [left-center],  we define:
\[\psi:   \bigcup_{t\in [0,1]} \phi_H^{-t}(\Sigma_{out})\to  \bigcup_{t\in [0,1]} \phi_H^t(\Sigma_{in})\; ,\]
such that  for  $t\in [0,1]$, the point $(\theta,r)=\phi_H^{-t}(\theta_0, r_0)$ with   $(r_0,\theta_0)\in \Sigma_{out}$ parametrized by  $(x,i)\in [-\eta^3,\eta^3]\times \Z_k$, is sent by $\psi$ to $\phi_{H}^{1-t}(\theta'_0, r'_0)$ with $(r'_0,\theta'_0)\in \Sigma_{in}$ parametrized by $(x,i)\in [-\eta^3,\eta^3]\times \Z_k$. 
 Note that the map $\psi$   is  analytic and  symplectic. Moreover it respects $H$.  
 Let $W^-/\cR $ be equal to $W^-$ quotiented by the equivalence relation $\cR$  spanned by:
  \[(r,\theta)\cR  (r',\theta')\quad  \text{if} \quad ( r,  \theta)= \psi( r', \theta'),\quad \forall 
 (r,\theta)\in  \bigcup_{t\in [0,1]} \phi_H^{-t}(\Sigma_{out})
\text{ and }
  (r',\theta')\in  \bigcup_{t\in [0,1]} \phi_H^t(\Sigma_{in})\; 
  .\] 
  Also as $\phi$ is an analytic diffeomorphism which
sends the intersection of its domain with $\partial W^-\setminus  \Sigma$  into $\partial W^-\setminus  \Sigma$,  the space  $W^-/\cR$  is an analytic surface with boundary. Moreover, we notice that $W^-/\cR$ is equal to the closed  disk $\bar \D$ without $k+1$  disks $(\D_i)_{0\le i\le k}$   as depicted in \cref{cap surgeryS} [center]:
\[W^-/\cR=  \bar \D \setminus (\bigcup_{i=0}^k \D_i)\]
This surgery respects the symplectic form $\Omega$ and the Hamiltonian $H$, which remains analytic on $W/\cR$ that we still denote  by $H$.    
 
Now we would like to   blow-down   the circles  $ \partial \D_i $ to fixed points $P_i$.   These blow-downs construct $k+1$ disks depicted in  \cref{cap surgeryS} [Center-Right].  In order to do so, we first observe that on $  \bigsqcup_i \partial \D_i$, the function $H$ is locally constant (equal to respectively $\eta^3$,  $-\eta^3$ and $\eta^3$) but $D H$ does not vanish on these circles. So we can apply \cref{euclidean coordinate}. For every $i$, it gives the existence of  a neighborhood $V_i$ of $\partial\D_i$ endowed with analytic and symplectic coordinate $\psi_\sigma: V_i\to \R/\Z\times [0,\delta]$  such that  $H\circ \psi^{-1}_i(\theta,r)= h_i(r)$  for an analytic map $h_i$.   So we can perform a blow-down of the circle $\partial \D_i$ to a fixed point  $P_i$. This blow-down sends $V_i$ to a disk of radius $\sqrt{ \delta/\pi}$ and on this disk
 the continuation of the   Hamiltonian is equal to $ P_i+(x,y)\mapsto h_i( x^2+y^2)$. 
As the unique critical points of $H|W^-$ were $(Q_i)_{i\in \Z_{2k}}$, these surgeries creates only $k+1$-new critical points at $P_i$ which are all with definite Hessian.    
\end{proof}
 
\bibliographystyle{alpha}
\bibliography{references}
\vskip 5pt

\begin{tabular}{l }
\emph{\normalsize Pierre Berger}\\
\small Institut de Math\'ematiques de Jussieu-Paris Rive Gauche, \\
CNRS,  Sorbonne Universit\'e, Paris, France \\
\small \texttt{pierre.berger@imj-prg.fr}\\

\end{tabular} 
\end{document}